\documentclass[11pt]{amsart}
\usepackage{amsmath}
\usepackage{cases}
\usepackage{amssymb, verbatim}
\usepackage{comment}

\title[Weak Geodesics for deformed Hermitian-Yang-Mills]{Weak Geodesics for the deformed Hermitian-Yang-Mills equation}
\author[A. Jacob]{Adam Jacob}
\thanks{Research supported in part by a Simons Collaboration Grant.}
\address{Department of Mathematics, University of California davis, 1 Shields Ave., Davis, CA 95616}
 \email{ajacob@math.ucdavis.edu}

\theoremstyle{plain}
\newtheorem{thm}{Theorem}[section]
\newtheorem{prop}[thm]{Proposition}
\newtheorem{defn}[thm]{Definition}
\newtheorem{lem}[thm]{Lemma}

\theoremstyle{definition}

\numberwithin{equation}{section}

\renewcommand{\thanks}[1]{\footnote{#1}} 

\newcommand{\be}{\begin{equation}}
\newcommand{\bea}{\begin{eqnarray}}
\newcommand{\eea}{\end{eqnarray}} \newcommand{\ee}{\end{equation}}
 
 \def\ba{\begin{eqnarray}}
\def\ea{\end{eqnarray}}

\def\log{\,{\rm log}\,}

\def\al{\alpha}

\def\ov{\overline}
\def\ti{\tilde}

\def\cF{{\mathcal F}}

\def\[{{\bf [}}
\def\]{{\bf ]}}

\def\pl{\partial}

\begin{document}

\maketitle

\begin{centering}

{\it Dedicated with gratitude to D.H. Phong on the occasion of his 65th birthday.}

\end{centering}

\vspace{.1in}

\begin{abstract}
We study weak geodesics in the space of  potentials for the deformed Hermitian-Yang-Mills equation. The geodesic equation can be formulated as a degenerate elliptic equation, allowing us to employ nonlinear Dirichlet duality theory,  as developed by  Harvey-Lawson. By exploiting the convexity of the level sets of the Lagrangian angle operator in the highest branch, we are able to construct $C^0$ solutions of the associated Dirichlet problem. 
\end{abstract}

\begin{normalsize}
\section{Introduction}

The goal of this paper is to formulate a theory of weak  geodesics in the space of potentials for the deformed Hermitian-Yang-Mills equation, and prove existence of a continuous solution to the corresponding Dirichlet problem. This  is motivated by recent progress on several geometric  PDEs relating existence to notions of stability arising from Geometric Invariance Theory. Most notably, the work of Chen-Donaldson-Sun \cite{CDS2,CDS3,CDS4}, which establishes existence of K\"ahler-Einstein metrics on K-stable Fano manifolds, can be connected to the theory of geodesics on the space of K\"ahler potentials developed by Donaldson \cite{D}, Mabuchi \cite{M} and Semmes \cite{S}. Specifically, recent work by  Berman-Boucksom-Jonsson \cite{BBJ} gives a different proof of the Chen-Donaldson-Sun theorem using the existence of $C^{1,\alpha}$ geodesics in the space of K\"ahler potentials, in relation to properness of certain functionals. 

Let $(X,\omega)$ be a compact K\"ahler manifold of complex dimension $n$, and let $[\alpha]\in H^{1,1}(X,\mathbb R)$ be a class with a fixed representative $\alpha$. The deformed Hermitian-Yang-Mills (dHYM) equation seeks a representative $\alpha_\phi\in[\alpha]$ solving:
 \be
 \label{DHYM1}
 {\rm Im}\left(e^{-i\hat\theta}(\omega+i\alpha_\phi)^n\right)=0
 \ee
for a fixed constant $\hat\theta$. This equation first appeared in the physics literature and was derived by considering open-string effective actions and studying BPS conditions \cite{MMMS}. Similar deformed equations also appeared in the works of N.C. Leung by looking at vector bundles over a symplectic manifold and considering moment maps of various differential forms on the space of connections \cite{Leung1,Leung2}. Later Leung-Yau-Zaslow considered equation \eqref{DHYM1} in the context of the semi-flat setup from SYZ mirror symmetry \cite{LYZ}. They proved solutions of \eqref{DHYM1} correspond via the Fourier-Mukai Transform to special Lagranagian graphs in the mirror manifold. This connection to special Lagrangians provides both insight into the structure of solutions to \eqref{DHYM1}, and inspiration that methods for studying \eqref{DHYM1} could shed light on the problem of existence of special Lagrangian submainfolds in general Calabi-Yaus.

A robust study of the dHYM equation in the K\"ahler case was undertaken in  \cite{CJY, JY}, and existence was  proven to be equivalent to an analytic class condition. Ideally, one hopes to replace this analytic condition with an  algebro-geometric stability condition. Recently, an outline of such an approach was given by Collins-Yau \cite{CY}, who derive algebraic obstructions to existence by constructing $C^{1,\alpha}$ geodesics on the space of potentials for $[\alpha]$. Because stability is not the focus of this paper, we direct the reader to   \cite{CY} (and references therein) for a thorough discussion of various notions of stability for  the dHYM equation and other related geometric equations. Instead, we concentrate on the analytic properties of the geodesic equation itself.

 Let $A$ be an annulus in $\mathbb C$. Let $\phi(s,z)$ be a real valued function defined on $M:=A\times X$, with the property that $\phi(s,z) = \phi(|s|,z)$.   Collins-Yau prove $\phi$ is a geodesic if it solves
\be
\label{geos0}
{\rm Im}\left( e^{-i\hat{\theta}}\left(\pi^{*}\omega + i(\pi^{*}\alpha +i D\ov{D} \phi)\right)^{n+1}\right)=0,
\ee
where $\pi$ is the projection onto the second factor of $A\times X$, and $D$ is the differential on $M$. Although this equation looks formally similar to \eqref{DHYM1}, it is in fact degenerate, as the pullback $\pi^{*}\omega$ is not a metric on $M$ and vanishes in the $s$ direction. As a result much care needs to be taken to compute the argument of the top form $\left(\pi^{*}\omega + i(\pi^{*}\alpha +i D\ov{D} \phi)\right)^{n+1}$. Our main result, Theorem \ref{continousgeo} in Section \ref{proof}, is to define a suitable notion of weak solution to \eqref{geos0} and produce $C^0$ solutions with given boundary data.  

Our study of \eqref{geos0} is inspired by the work of Rubinstein-Solomon \cite{RS}, and later Darvas-Rubinstein \cite{DR}, on geodesics in the space of Lagrangian potentials. Both study a similar degenerate elliptic equation, and Rubinstein-Solomon apply the Dirichlet  duality theory of Harvey-Lawson \cite{HL1} to define weak solutions and  solve the corresponding Dirichlet problem. Our equations are formally quite similar, and in our case we can follow \cite{RS} to define a similar space-time Lagrangian angle related to \eqref{geos0}. However, we encounter some major differences in extending their work to solutions of the Dirichlet problem. First and foremost, we work in the complex setting as opposed to a real setting. Here, the work Y. Wang \cite{WYu} (see also  \cite{TWWY}) demonstrates how to adapt real PDE theory  to the case of symmetric matrices which are invariant under  the complex structure.  We are also forced to work on $A\times X$, where $A$ is an annular domain as opposed to a real interval, although this does not pose any serious issues.

The main difficulty we encounter stems from the fact  that our cross section $X$ is  a compact manifold as opposed to a Euclidean  domain. Fortunately, Harvey-Lawson's Dirichlet duality theory extends to compact manifolds \cite{HL2}, which we employ. However, unlike  Euclidean domains, by the maximum principle  compact manifolds do not admit non-trivial strictly subharmonic functions, which can be used to approximate subsolutions to \eqref{geos0} from above. This approximation plays a fundamental role in the comparison principle for the Dirichlet problem and is a key part of the existence proof in \cite{RS}. Instead, in our case we use the fact that the level sets for the angle operator associated to \eqref{DHYM1} are convex in the highest branch, and construct a different approximation using this convexity. Because we employ a relatively general argument, we hope that our techniques carry over to other equations on compact manifolds.

Our paper is organized as follows. In Section \ref{background} we introduce the dHYM equation and derive the corresponding geodesic equation on the space of potentials. In Section \ref{sub} we review the basics of Dirichlet duality theory. We reformulate both the dHYM equation and the geodesic equation into the framework of  Dirichlet duality  theory  in Section \ref{angles}. Finally, Section \ref{proof} contains the proof of Theorem \ref{continousgeo}.

\vspace{\baselineskip}

{\bf Acknowledgements.} We would like to express our gratitude to Tristan C. Collins for many valuable  discussions and comments. We also thank the referee for  helpful suggestions. This work was funded in part by a Simons collaboration grant. We would also like to thank the Institute for Mathematical Sciences at the National University of Singapore, where some of this work took place, for providing a stimulating research environment.

\section{The deformed Hermitian-Yang-Mills equation and geodesics}
\label{background}

Let $(X,\omega)$ be a compact K\"ahler manifold, and let $[\alpha]\in H^{1,1}(X,\mathbb R)$ be a class with a fixed representative $\alpha$. By the $\partial\bar\partial$-lemma, any other representative  can be expressed as $\alpha_\phi=\alpha+i\partial\bar\partial \phi$, for a given potential $\phi$.  We are interested in the problem of finding a $\phi\in C^\infty(X)$ so that the complex function
 \be
\frac{(\omega+{i}\al_\phi)^n}{\omega^n}:X\longrightarrow \mathbb C\nonumber
 \ee
 has constant argument. For this to be possible, one necessary condition is that  the integral
  \be
 Z_{X}:=\int_X{(\omega+{i}\al)^n}\nonumber
 \ee
 lie in $\mathbb C^*$, which we always assume. Note that $Z_X$ only depends on the classes $[\omega]$ and $[\alpha]$, and not on any particular choice of representative.  Setting $\hat\theta:={\rm arg}(Z_X)\in[0,2\pi)$,  the deformed Hermitian-Yang-Mills equation (dHYM)  seeks a function $\phi$ solving
 \be
 {\rm Im} \left({e^{-i\hat\theta}(\omega+{i}\al_\phi)^n}\right)=0.\nonumber
 \ee

This equation can be reformulated as follows. First, define the relative endomorphism $\Lambda_\phi:=\omega^{-1}\alpha_\phi$ which acts on $T^{1,0}X$. Fixing a point $p\in X$ and choosing coordinates so $\omega(p)_{\bar kj}=\delta_{\bar kj}$ and $\alpha_\phi{}_{\bar kj}(p)=\lambda_j\delta_{\bar kj}$,  we can rewrite  our complex function as
\be
\frac{(\omega+{i}\al_\phi)^n}{\omega^n}={\rm det}(Id_{T^{1,0}X}+{i}\Lambda_\phi)=\prod_{k=1}^n(1+i\lambda_k).\nonumber
\ee
The modulus and argument   are given by
\be
r(\alpha_\phi)=\sqrt{\prod_{k=1}^n(1+\lambda_k^2)}\qquad{\rm and}\qquad \Theta(\alpha_\phi)=\sum_{k=1}^n{\rm arctan}(\lambda_k).\nonumber
\ee
By the arctan formulation we see that the angle $\Theta(\alpha_\phi)$ lives in $\mathbb R$ rather than $S^1$. Due to analogy with the case of special Lagrangian graphs (see \cite{HL, CNS}) we refer to $\Theta(\alpha_\phi)$ as the {\it Lagrangian angle}. The dHYM equations can now be expressed as
\be
\label{DHYM2}
\Theta(\alpha_\phi)=\sum_{k=1}^n{\rm arctan}(\lambda_k)=c,\nonumber
\ee
where $c$ is a constant which satisfies $c=\hat\theta$ mod $2\pi$.  A choice of  $c$ is referred to as  {\it branch} of the dHYM equation. Note that if a solution to the dHYM equation exists, the choice of a branch is specified. In general, if one is only given a cohomology class  $[\alpha]$, determining which branch of the equation to work with is a difficult question, as there is no natural lift of ${\rm arg}(Z_X)$ to $\mathbb R$. In this paper, the branch is always specified by assuming an initial analytic condition on $[\alpha]$. Namely, we assume that for fixed representative $\alpha$, the oscillation of the lifted angle $\Theta(\alpha)$ is bounded by $\pi$, which then uniquely specifies which branch $c$ to choose.

The corresponding geodesic equation is the main focus of this paper, which we now introduce. First, consider the space of ``positive" $(1,1)$ forms in $[\alpha]$, which is parametrized by potentials
\be
\label{space}
\mathcal{H} = \{ \phi \in C^{\infty}(X,\mathbb{R}) : {\rm Re}\left(e^{-i\hat{\theta}}(\omega+i\alpha_{\phi})^n\right)>0\}.
\ee
We assume that $\mathcal{H}$ is non-empty. Now, a choice of $\phi\in\mathcal{H}$ gives a corresponding lifted angle $\Theta(\alpha_\phi)$, and using the definition of $\mathcal H$ we see the oscillation of this angle is bounded by $\pi$, which specifies a  choice of a branch $c$ for equation \eqref{DHYM2}.

At  $\phi\in\mathcal{H}$  the tangent space is again  $ C^{\infty}(X,\mathbb{R})$, allowing us to define the following natural  metric.   For $\psi_i \in T_{\phi}\mathcal{H}$, $i =1,2$, set
$$
\langle \psi_1, \psi_2\rangle = \int_{X} \psi_1 \psi_2 {\rm Re}\left(e^{-i\hat{\theta}}(\omega+i\alpha_{\phi})^n\right).
$$
It is with respect to this metric that we consider the following geodesic equation.
\begin{lem}
Let $\phi(t):[0,1]\rightarrow{\mathcal H}$ be a smooth curve. Then $\phi(t)$ is a critical point for arc length if and only if it solves 
\be
\label{geo1}
\ddot{\phi}{\rm Re}\left(e^{-i\hat{\theta}}(\omega+i\alpha_{\phi})^n\right)+ in\pl\dot{\phi}\wedge \ov{\pl}\dot{\phi}\wedge{\rm Im}\left(e^{-i\hat{\theta}}(\omega+i\alpha_{\phi})^{n-1} \right)=0,
\ee
where $\dot\phi$ is used to denote a derivative in time.
\end{lem}

The proof of the above lemma is a straightforward computation, and all details can be found in Proposition 2.7 in \cite{CY}. Now, similar to the study geodesics in the space of K\"ahler potentials, it is helpful to reformulate \eqref{geo1} as a degenerate elliptic equation on an annular domain in complex dimension $n+1$, as opposed to an equation on $X\times[a,b]$ for some interval $[a,b]\subset\mathbb R$.  
\begin{lem}
\label{reformulation1}
Let $A$ be a annulus in $\mathbb C$. Let $\phi(s,z)$ be a real valued function defined on $M:=A\times X$, with the property that $\phi(s,z) = \phi(|s|,z)$.  Then $\phi(e^{-t},z)$ is a geodesic solving \eqref{geo1} if and only if
\be
\label{geos}
{\rm Im}\left( e^{-i\hat{\theta}}\left(\pi^{*}\omega + i(\pi^{*}\alpha +i D\ov{D} \phi)\right)^{n+1}\right)=0
\ee
\be
\label{geos5}
{\rm Re}\left( e^{-i\hat{\theta}}\left(\pi^{*}\omega + i(\pi^{*}\alpha+i D\ov{D} \phi)\right)^{n}\right)|_{X}> 0,
\ee
where $D$ denotes the differential in $A\times X$, and $\pi$ is the  projection onto the second factor.  The second constraint \eqref{geos5} expresses that $\phi \in \mathcal{H}$.
\end{lem}
\begin{proof}
This computation is contained in Lemma 2.8 in \cite{CY}, and we include the details here for convenience. For simplicity we exclude the pullback $\pi^*$ from our notation. We use $D$  to denote the differential in $X\times A$, $\partial_s$ for differential on $A$, and $\partial$ for the differential on $X$. Using  $t = -\log|s|$, we have
\be
\label{sderiv}
\pl_{s}\phi = \frac{-ds}{s}\dot{\phi}, \quad \bar\pl_{s}\phi = \frac{-d\bar s}{\bar{s}}\dot{\phi}, \quad \pl_s\bar\pl_{s}\phi = \frac{ds\wedge d\bar s}{|s|^2}\ddot{\phi}.
\ee
Now
\be
\begin{aligned}
e^{-i\hat{\theta}}\left(\omega + i(\alpha +i D\ov{D} \phi)\right)^{n+1} &= i (i\pl_s\bar\pl_{s}\phi) \wedge e^{-i\hat{\theta}}(\omega+i\alpha_{\phi})^n\\
&-n i\pl_s\bar \pl \phi\wedge i \pl \bar\pl_{s}\phi \wedge e^{-i\hat{\theta}}(\omega+i\alpha_{\phi})^{n-1}.\nonumber
\end{aligned}
\ee
Because $i\pl_s\bar\pl_{s} \phi$ is real, we can write
\be
{\rm Im}\left(\, i (i\pl_s\bar\pl_{s}\phi) \wedge e^{-i\hat{\theta}}(\omega+i\alpha_{\phi})^n\right) = (i\pl_s\bar\pl_{s}\phi)\wedge {\rm Re}\left( e^{-i\hat{\theta}}(\omega+i\alpha_{\phi})^n \right).\nonumber
\ee
Additionally, $i\pl_s\pl \phi\wedge i \pl \bar\pl_{s}\phi$ is real, which gives
\bea
&&{\rm Im}\left(n i\pl_s\bar \pl \phi\wedge i \pl \bar\pl_{s}\phi \wedge e^{-i\hat{\theta}}(\omega+i\alpha_{\phi})^{n-1}\right) \nonumber\\&&= n i\pl_s\bar \pl \phi\wedge i \pl\bar\pl_{s}\phi \wedge{\rm Im}\left( e^{-i\hat{\theta}}(\omega+i\alpha_{\phi})^{n-1}\right).\nonumber
\eea
Furthermore, the mixed terms can be expressed as
\be
i\pl_s\bar \pl \phi\wedge i \pl \bar\pl_{s}\phi = -\frac{1}{|s|^2} i\pl \dot{\phi}\wedge \ov{\pl}\dot{\phi} \wedge ids\wedge d\ov{s}.\nonumber
\ee
Putting everything together gives
\be
{\rm Im}\left(e^{-i\hat{\theta}}\left(\pi^{*}\omega + i(\pi^{*}\alpha_0 +i D\ov{D} \phi)\right)^{n+1} \right)= \nonumber
\ee
\be
\left[\ddot{\phi} {\rm Re}\left(e^{-i\hat{\theta}}(\omega+i\alpha_{\phi})^n\right)+n i\pl \dot{\phi}\wedge \ov{\pl}\dot{\phi} \wedge{\rm Im}\left( e^{-i\hat{\theta}}(\omega+i\alpha_{\phi})^{n-1}\right)\right]\wedge \frac{ ids\wedge d\ov{s}}{|s|^2}.\nonumber
\ee
This completes the proof of the lemma.
\end{proof}

To conclude this section, note that if we assume $\alpha_0>0$ and take the ``small radius limit", that is, we consider $t \omega$ and take $t\rightarrow 0$, then by direct computation we have $ e^{i\hat{\theta}} \rightarrow (i)^n.$ This implies
\be
{\rm Im}\left( e^{-i\hat{\theta}}\left(\pi^{*}\omega + i(\pi^{*}\alpha_0 +i D\ov{D} \phi)\right)^{n+1}\right) \rightarrow \left(\pi^{*}\alpha_0 + iD\ov{D} \phi\right)^{n+1},\nonumber
\ee
which is the usual geodesic equation of Donaldson-Mabuchi-Semmes \cite{D,M,S}.

\section{Subequations and subharmonic functions}
\label{sub}

Here we review the Dirichlet-Duality theory of Harvey Lawson, which is our main tool for studying \eqref{geos}. In particular we introduce several relevant definitions, and detail the appropriate notion of a weak  solution to the Dirichlet problem for our setup. The majority of this content can be found in \cite{HL1, HL2}, yet we include the details here for the reader's convenience.

Let $M$ be a compact Riemannian manifold of dimension $m$. Denote by $J^2(M)$ the bundle of 2-jets over $M$. At a point $p\in M$ the fiber of this bundle is given by
 \be
 J^2_p(M)=C^\infty_p/C^\infty_{p,3}.\nonumber
 \ee
 Here $C^\infty_p$ denotes the germs of smooth functions, and $C^\infty_{p,3}$ is the subspace of germs that vanish to order 3. There is a short exact sequence of bundles
 \be
 \label{sequence}
 0\rightarrow {\rm Sym}^2(T^*M)\rightarrow J^2(M)\rightarrow J^1(M)\rightarrow0,
 \ee
 where $ {\rm Sym}^2(T^*_pM)$ is the space of symmetric bilinear forms on $T_pM$. As mentioned in \cite{HL2}, the above sequence does not split naturally, however, on a Riemannian manifold the Hessian of a function can be defined using the Riemannian metric, which in turn leads to a splitting.
 
 Let $F\subset J^2(M)$ be an arbitrary subset of the 2-jet bundle.  A function $u\in C^2(M)$ is called {\it $F$-subharmonic} if its 2-jet satisfies
 \be
 J^2_pu\in F_p\nonumber
 \ee
for all $p\in M$, and {\it strictly $F$-subharmonic} if on each fiber its 2-jet lies in  int$(F_p)$. Given this  setup, Harvey-Lawson introduce a special class of subsets $F\subset J^2(M)$, called {\it subequations}, for which $F$-subharmonic functions behave in many ways similar to classical subharmonic functions on Euclidean space. Furthermore, using this class, they develop a  theory building towards  a general solution of the Dirichlet problem.

For the purposes of this paper, we do not give the full definition of a subequation, and instead restrict to a special case (for the general case we direct the   reader to  \cite{HL1, HL2}). Specifically, we say that $F$ is {\it Dirichlet set} if it satisfies the  condition:
\be
\label{(P)}
F+\mathcal P\subset F,\nonumber
\ee
where $\mathcal P$ is the set of non-negative symmetric matrices.
Next, we say $F$ is of {\it purely second order} if, with respect to the splitting of \eqref{sequence} given by the Riemannian metric, one can write
\be
F=\mathbb R\oplus T^*M\oplus F'\nonumber
\ee
for some $F'\subset {\rm Sym}^2(T^*_pM)$. In other words, if $u\in C^2(M)$, then $J^2_p u\in F$ if and only if ${\rm Hess}_p u\in F'$. Given a purely second order set $F$, Harvey-Lawson show $F$ is a subequation if and only if  it is a Dirichlet set, and for any fiber over $p$ it holds ${\rm int}(F_p')=({\rm int}F')_p$. For the purposes of this paper we will use this condition to define our subequations.

We now extend the definition of  $F$-subharmonic functions from $C^2(M)$ to the space of all upper semi-continuous functions  ${\rm USC}(M):M\rightarrow[-\infty,\infty)$. 

\begin{defn}
A function $u\in{\rm USC}(M)$ is called $F$-subharmonic if for each point $p\in M$, and each $C^2$ function $\phi$ near $p$ with $u(p)=\phi(p)$, one has
\be
u-\phi\leq 0\,\implies J^2_p\phi\in F_p.\nonumber
\ee
We denote this space by $F(M)$.
\end{defn}
Note that if $u\in C^2 (M)$, then we can choose $u$ as our test function to recover the definition of a $F$-subharmonic function given above. In fact, if  $F$ is a subequation, one can check
\be
u\in F(M)\cap C^2(M)\iff J^2_pu\in F_p\,\,\,{\rm for\,\,all}\,\,p\in M.\nonumber
\ee
Next we define strictly $F$-subharmonic functions  in USC$(M)$. For $\delta>0$ define the set $F^\delta\subset F$ by
\be
\label{strictsub}
F^\delta_p:=\{A\in F_p\,|\,{\rm dist}(A,\sim F_p)\geq \delta\}.
\ee
Here $\sim F_p$ denotes the complement of   $F_p$ in a given fiber. 
\begin{defn}
A function $u\in{\rm USC}(Y)$ is said to be strictly  $F$-subharmonic if for each point $p\in Y$, there is a neighborhood $U$ of $p$ and a constant $\delta>0$ so that $u$ is $F^\delta$ subharmonic on $U$.
\end{defn}

Given a subset $F\subset J^2(M)$, the {\it Dirichlet dual} of $F$ is defined as
\be
\ti F:=\sim(-{\rm int}(F))=-(\sim{\rm int}(F)).\nonumber
\ee
Proposition 3.10 in \cite{HL2} shows that $F$ is a subequation if and only if $\ti F$ is a subequation. The notion of a Dirchlet dual allows us to define $F$-subhamonic functions, similar to the observation that in the classical theory a function is harmonic if and only if it is subharmonic and superharmonic. 
\begin{defn}
\label{harmonicdef}
A function $u\in{\rm USC}(M)$ is said to be $F$-harmonic if
\be
u\in F(M) \qquad{\rm and}\qquad-u\in\ti F(M).\nonumber
\ee
\end{defn}
Note that if $u$ is $F$-harmonic and in $C^2(M)$, then  $J^2_pu\in\partial F_p$ for all $p\in M$. Finally, we introduce the appropriate notion of a a weak solution to the Dirichlet problem. Suppose $M$ is a manifold with boundary. 

\begin{defn}
Fix $\phi\in C^0(\partial M)$. We say $u\in C^0(\overline M)$ solves the Dirichlet problem for $F$ if it satisfies
\be
(1) \,u\,\,{\rm is}\,\,F{\text-harmonic}{\rm\,\, on\,\, int}(M)\qquad{\rm and}\qquad (2)\,u|_{\partial M}=\phi.\nonumber
\ee 
\end{defn}
Using this definition we solve the Dirichlet problem for  \eqref{geos}. Our next step is to reformulate this equation into the language of subequations.

\section{Subequations for two types of Lagrangian angles} 
\label{angles}

Let $M:=A\times X$, where $A:=\{s\in\mathbb C\,|\, 1\leq |s|\leq 2\}$, and $X$ is our given compact K\"ahler manifold.  The main goal of this section is to define a subequation $\mathcal F_{\sigma,c}\subset J^2(M)$ for which $\mathcal F_{\sigma,c}$-harmonic functions satisfy \eqref{geos}.  As a first step, we define a subequation for the standard Lagrangian angle $\Theta(\alpha_\phi)$ on $X$. 

To begin, we define a subequation on the space of symmetric matrices  ${\rm Sym}^{2}(\mathbb R^{2n})$ in Euclidean space, and then extend this definition to the 2-jet bundle of $M$. Let $\omega$ be the standard K\"ahler form on $\mathbb C^n=\mathbb R^{2n}$, and $J$ the standard complex structure. Since Dirichlet-Duality theory is formulated in terms of symmetric matrices,   in our complex setting we need to first identify ${\rm Herm}(\mathbb C^n)$ with the subset of  ${\rm Sym}^{2}(\mathbb R^{2n})$ given by $J$-invariant matrices. 

Any Hermitian matrix $H$ can be written as $H=A_1+i A_2,$ where $A_1$ is a real symmetric matrix and $A_2$ is a real skew-symmetric matrix. Specifically, set $A_1=\frac12(H+\bar H)$ and $A_2=\frac1{2i}(H-\bar H).$ We now define the inclusion:
\be
\iota(H):=  \begin{pmatrix} A_1 & A_2  \\ -A_2 & A_1 \end{pmatrix}\in {\rm Sym}^2(\mathbb R^{2n}).\nonumber
\ee
Moreover, if $N\in  {\rm Sym}^2(\mathbb R^{2n})$, then the projection onto the $J$-invariant part is given by
\be
p(N):=\frac12(N+J^TNJ).\nonumber
\ee
For $A\in{\rm Sym}^{2}(\mathbb R^{2n})$, the angle of a symmetric matrix is given by
\be
\label{thetadef}
\ti\Theta (A):=\frac12{\rm tr\,\,arg}(Id_{\mathbb R^{2n}}+i \,\,p(A)).
\ee
Note that this function is valued in $(-n\frac\pi2,n\frac\pi2)$. The factor of $1/2$  is included to account for the fact that matrices in the image of $p(\cdot)$ have eigenvalues of multiplicity $2$, with corresponding eigenvectors $v_i$ and $J(v_i)$. 

For any $c\in\mathbb R$, we set
\bea
F_c&:=&\{A\in {\rm Sym}^{2}(\mathbb R^{2n})\,|\, \ti\Theta(A)\geq c\}.\nonumber
\eea
The following Lemma is proven in \cite{CNS,HL1,HL2}. 
\begin{lem}
\label{localangle}
When $|c|<\frac{n\pi}2$, the set $F_c$ defines a local subequation on $\mathbb R^{2n}$. Furthermore, its dual is given by
\be
\ti F_c=F_{-c}.\nonumber
\ee
\end{lem}
Given the above subequation on ${\rm Sym}^{2}(\mathbb R^{2n})$, there is a natural extension to $J^2(X)$, using a   bundle automorphism. Again let $\omega$ be the K\"ahler form on $X$. We follow the construction from Section 6 of \cite{HL2}.  
 
Recall the Hermitian endomorphism $\Lambda=\omega^{-1}\alpha$  from Section \ref{background}. Since $\omega$ determines a metric on $X$, the $2$-jet bundle splits as
\be
 J^2(X)={\mathbb R}\oplus T^*X\oplus {\rm Sym}^2(T^*X).\nonumber
 \ee
 Using this splitting, consider the bundle automorphism
 \be
 \sigma(r,\ell,A)=(r,\ell, A+\iota(\Lambda)(x)),\nonumber
 \ee
and set $F_{\sigma,c}:=\sigma^{-1}(F_c)$. Then we have 
 \be
 A\in F_{\sigma,c}\,\iff\, \ti\Theta(\iota(\Lambda)+A)\geq c.\nonumber
 \ee
 It then follows that $ F_{\sigma,c}$-harmonic functions are functions $u$ solving
 \be
  \ti\Theta(\iota(\Lambda)+{\rm Hess}\,u)= c,\nonumber
 \ee
 which we can  write as
 \be
\frac12\, {\rm tr\,\,arg}\left(Id_{T^*X}+i \left(\iota(\Lambda)+p({\rm Hess}\,u)\right)\right)=c.\nonumber
 \ee
 Note we used that $p(\iota(\Lambda))=\iota(\Lambda)$.  The above equation holds pointwise on $X$, and working in normal coordinates one can see it is   equivalent  to \eqref{DHYM2}.

To define the dual subequation $\ti F_{\sigma,c}$, consider the bundle automorphism
 \be
 \ti \sigma(r,\ell,A)=(r,\ell, A-\iota(\Lambda)(x)).\nonumber
 \ee
By Lemma \ref{localangle}, the dual subequation to $F_c$ is $F_{-c}$, and so we define  the subequation $F_{\ti\sigma,-c}:=\ti\sigma^{-1}( F_{-c})$. It then follows that
 \be
 A\in F_{\ti \sigma,-c}\,\iff\, \ti\Theta(-\iota(\Lambda)+A)\geq -c.\nonumber
  \ee
One can now check, using the argument from Lemma 6.14 in \cite{HL2}, that $F_{\ti \sigma,-c}$ is indeed the dual subequation to  $F_{\sigma,c}.$ Summing up, we have proved the following:

\begin{lem}
\label{localangle2}
When $|c|<\frac{n\pi}2$, the set $F_{\sigma,c}\subset J^2(X)$ defines a global subequation, with its dual is given by $F_{\ti \sigma,-c}$. Furthermore, $ F_{\sigma,c}$-harmonic functions are weak solutions to the deformed Hermitian-Yang-Mills equation \eqref{DHYM2}.
\end{lem}

Next we turn to the definition of the space-time Lagrangian angle, following \cite{RS}. As above, we first consider the local case, and then use jet equivalence to define a subequation on $J^2(M)$. Working on  $\mathbb C^{n+1}=\mathbb R^{2n+2}$, let $I_{2n}$ be the diagonal matrix with 2n+2 entries given by  ${\rm diag}(0,0,1,...,1)$. For a given $A\in{\rm Sym}^2(\mathbb R^{2n+2})$, if det$(I_{2n}+i p(A))\neq0$, then we define the space-time angle of $A$ by
\be
\Phi(A):={\rm arg\,det}(I_{2n}+i p(A))^{\frac12}\in S^1.\nonumber
\ee
Similar to the case of the standard Lagrangian angle, if we let $\{\mu_i\}$ denote the eigenvalues of  $B=I_{2n}+i \, p(A)$, then (by a slight abuse of notation), the angle lifts to $\mathbb R$ viaDi
\be
\label{spaceangle2}
\Phi(A)=\sum_{i=1}^{2n+2}{\rm arg}(\mu_i).
\ee

We need to extend this definition to the case where the determinant vanishes. Write the matrix $p(A)$ as $(a_{ij})_{i,j=1}^{n+2}$, and define the truncated matrix
\be
p(A)^+=(a_{ij})_{i,j=3}^{n+2}.\nonumber
\ee
Let $\vec a_1=(a_{13},a_{14},...,a_{1(2n+2)})$, so the first row of $p(A)$ is given by $(a_{11},0,\vec a_1)$ (the second entry of this row must be zero by $J$ invariance). Define the set
\be
\mathcal S=\{A\in {\rm Sym}^2(\mathbb R^{2n+2})|\,p(A)={\rm diag}(0,p(A)^+)\}.\nonumber
\ee
Away from this set, $\Phi$   is well defined. To see this, denote $B_0:=I_{2n}+i p(A)$, which corresponds to the case of $\eta=0$ from Lemma \ref{matrixfun} below. Equation \eqref{detB} now gives that det$(B_0)$ is non-zero as long as both $a_{11}$ and $\vec a_1$ do not vanish, and so $\Phi$  is well defined away from $\mathcal S$.

 We now use  the following upper semi-continuous extension to define the angle away from $\mathcal S$:
\be
\ti\Phi(A) = \begin{cases} \qquad\qquad\,\,\,\Phi(A) &\mbox{if } A\notin \mathcal S \\
\frac\pi2+\frac12{\rm tr\,\,arg}(I_{\mathbb R^{2n}}+i p(A)^+)& \mbox{if } A\in \mathcal S. \end{cases}\nonumber
\ee
Similarly the lower semi-continuous extension is defined as:
\be
\underline{\ti\Phi}(A) = \begin{cases} \qquad\qquad\,\,\,\Phi(A) &\mbox{if } A\notin \mathcal S \\
-\frac\pi2+\frac12{\rm tr\,\,arg}(I_{\mathbb R^{2n}}+i p(A)^+)& \mbox{if } A\in \mathcal S. \end{cases}\nonumber
\ee

Before we use the lifted space-time angle $\ti\Phi(\cdot)$ to define a subequation, we need a few preliminary results. These results are extensions of the work in \cite{RS} to the Hermitian case. Therefore, we only presented differences that need to be considered, and simply restate those results which carry over to our setting with no modification. 

\begin{lem}
If $A\in{\rm Sym}^2(\mathbb R^{2n+2})$ and $\delta>0$, then ${\rm Re}((\delta Id_{\mathbb R^{2n}}+i A)^{-1})$ is positive definite.
\end{lem}

\begin{lem}
\label{matrixfun}
Let $A\in {\rm Sym}^2(\mathbb R^{2n+2})$, and consider the matrix $I^\eta_{2n}$ given by the diagonal $(\eta,\eta,1,1,...,1)$.  The eigenvalues $\mu_i$ of $B_\eta=I^{\eta}_{2n}+i p(A)$ satisfy Re$(\mu_i)\geq 0$. Furthermore, if $\vec a_1\neq 0$ or $\eta>0$, then Re$(\mu_i)>0$.
\end{lem}
\begin{proof}
First we identify $B_\eta$  with a matrix in Herm$(\mathbb C^{n+1})$. Let $A^{\mathbb C}$ denote the  Hermitian matrix satisfying $\iota (A^{\mathbb C})=p(A)$. Specifically, choose coordinates on $\mathbb R^{2n+2}$ so 

\be
J:=  \begin{pmatrix} 0 & -Id_{\mathbb R^{n+1}}  \\ Id_{\mathbb R^{n+1}} & 0 \end{pmatrix}\qquad{\rm and}\qquad p(A)=\begin{pmatrix} A_1 & A_2  \\ -A_2 & A_1 \end{pmatrix}.\nonumber
\ee
Then $A^{\mathbb C}=A_1+i A_2$. We define
\be
B^{\mathbb C}_\eta=I^\eta_{n}+i A^{\mathbb C},\nonumber
\ee
and note that $\iota( B^{\mathbb C}_\eta)=B_\eta$. Let $B^{\mathbb C}_+$ be the $n\times n$ Hermitian matrix given by removing the first row and column  from $B^{\mathbb C}$, and define $A^{\mathbb C}_+$ in the same fashion. $B^{\mathbb C}_+$ is invertible, since if $\lambda_i$ are the real eigenvalues of $A^{\mathbb C}_+$, then the eigenvalues for $B^{\mathbb C}_+$ are $1+i \lambda_i$, which are always non-vanishing. Furthermore, the real part of the eigenvalues of $(B^{\mathbb C}_+)^{-1}$ are given by $1/(1+\lambda_i^2)$, and are thus all positive. 

 Let $({\bf a}_{11},\vec {\bf a}_1)$ is the first column of $A^{\mathbb C}$, from which we see  $(\eta+i{\bf a}_{11},i \vec {\bf a}_1)$ is the first column of $B^{\mathbb C}_\eta$. We can write the determinant of $B^{\mathbb C}_\eta$ as 
\be
\label{detB}
{\rm det} B^{\mathbb C}_\eta={\rm det} B^{\mathbb C}_+\left(\eta+i{\bf a}_{11}+\vec{\bf a}_1 (B_+^{\mathbb C})^{-1} {\vec{ \bf a}}_1^*\right),
\ee
and so  the determinant  vanishes only if $\left(\eta+i {\bf a}_{11}+\vec{\bf a}_1 (B_+^{\mathbb C})^{-1} {\vec{ \bf a}}_1^*\right)$ vanishes. The eigenvalues of $B^{\mathbb C}_+$ all have strictly positive real part, so we concentrate on 
\be
{\rm Re}\left(\eta +i {\bf a}_{11}+\vec{\bf a}_1 (B_+^{\mathbb C})^{-1} {\vec{ \bf a}}_1^*\right)=\eta+{\rm Re}\left(\vec{\bf a}_1 (B_+^{\mathbb C})^{-1} {\vec{ \bf a}}_1^*\right).\nonumber
\ee
Note that $\vec {\bf a}_1$ has both a real and imaginary part. However, as before if we choose coordinates so  $B^{\mathbb C}_+$  is diagonal with eigenvalues $1+i \lambda_i$, then  
\be
{\rm Re}\left(\vec{\bf a}_1 (B_+^{\mathbb C})^{-1} {\vec{ \bf a}}_1^*\right)={\rm Re}\left(\sum_{i=2}^{n+1}\frac{|{\bf a}_{1i}|^2}{1+i \lambda_i}\right)=\sum_{i=2}^{n+1}\frac{|{\bf a}_{1i}|^2}{1+\lambda_i^2}.\nonumber
\ee
Thus, it follows that
\be
{\rm Re}\left(\eta +\vec{\bf a}_1 (B_+^{\mathbb C})^{-1} {\vec{ \bf a}}_1^*\right)\geq0,\nonumber
\ee
with strict inequality if either $\eta$ or $\vec{\bf a}_1$ are non-zero. This proves the conclusion for $B^{\mathbb C}_\eta$, and thus for $B_\eta$ by the relationship $B_\eta=\iota (B^{\mathbb C}_\eta)$.
\end{proof}

The above lemma allows us to define the argument of $B=I_{2n}+i p(A)$, using Corollary 3.5 from \cite{RS}. This, in turn, allows us to expressed the lifted angle as
\be
\label{liftedspacetimeanlge}
\ti\Phi(A) = \begin{cases}\,\,\,\,\,\, \frac12{\rm tr\,\,arg}(I_{2n}+i p(A))&\mbox{if } A\notin \mathcal S \\
\frac\pi2+\frac12{\rm tr\,\,arg}(I_{\mathbb R^{2n}}+i p(A)^+)& \mbox{if } A\in \mathcal S. \end{cases}
\ee
A similar formula holds  for $\underline{\ti\Phi}(A)$). Note that, away from ${\mathcal S}$, both $\ti\Phi(A)$ and $\underline{\ti\Phi}(A)$ are differentiable functions. Thus we have:
\begin{prop}
 $\Phi$ is a   differentiable function on  ${\rm Sym}^2(\mathbb R^{2n+2})\backslash\mathcal S$. Furthermore, ${\ti\Phi}$ and $\underline{\ti\Phi}$ are the smallest and largest upper and lower semi-continuous functions extending $\Phi$ over $\mathcal S$. 
\end{prop}

At last we can define our Dirichlet set for the lifted Lagrangian angle. Consider the set
\be
\cF_c:=\{A\in{\rm  Sym}^2(\mathbb R^{2n+2})\,|\,\ti\Phi(A)\geq c\}.\nonumber
\ee
The following result is a complex analogue of Theorem 5.1 from \cite{RS}. The proof follows in a similar fashion.
\begin{prop}
\label{localsub}
If $|c|<(n+1)\frac\pi2$, then $\cF_c\subset {\rm  Sym}^2(\mathbb R^{2n+2})$ is closed and non-empty. Additionally $\cF_c=\overline{{\rm Int} \cF_c}$. Furthermore, $\cF_c$ is  Dirichlet set in the sense of \eqref{(P)}, with its dual given by $\ti\cF_c=\cF_{-c}.$
\end{prop}

We now define our desired subequation on $M$ using local jet-equivalence. The 2-jet bundle on $M$ decomposes as
 \be
 J^2(M)={\mathbb R}\oplus T^*M\oplus {\rm Sym}^2(T^*M).\nonumber
 \ee
Consider the bundle automorphism $ \sigma: J^2(M)\rightarrow  J^2(M)$ defined by
 \be
 \sigma(r,\ell,A)=(r,\ell, A+\iota(\Lambda)(x)).\nonumber
 \ee
Set $\mathcal F_{\sigma,c}|_U:=\sigma^{-1}(\mathcal F_c)$. We then have 
 \be
 A\in \mathcal F_{\sigma,c}\,\iff\, \ti\Phi(\iota(\Lambda(p))+A_p)\geq c.\nonumber
 \ee
at each point $p\in M$, using that $p(\iota(\Lambda))=\iota(\Lambda)$.

To identify the dual of $\mathcal F_{\sigma,c}$, we use that $\mathcal F_{-c}$ is the local dual of $\mathcal F_{c}$. Consider the bundle automorphism
 \be
 \ti \sigma(r,\ell,A)=(r,\ell, A-\iota(\Lambda)(x)),\nonumber
 \ee
and set $\mathcal F_{\ti\sigma,-c}|_U:=\ti\sigma^{-1}(\mathcal F_{-c})$. We then have 
 \be
 A\in {\mathcal F}_{\ti\sigma,-c}\,\iff\, \ti\Phi(-\iota(\Lambda(p))+A_p)\geq -c\nonumber
 \ee
at each point $p\in M$. Just as above, we apply Lemma 6.14 in \cite{HL2} to conclude that ${\mathcal F}_{\ti\sigma,-c}$ is the dual of $\mathcal F_{\sigma,c}$.  In conclusion we have:
\begin{prop}
\label{localangle3}
For $|c|<\frac{(n+1)\pi}2$, the set $\mathcal F_{\sigma,c}\subset J^2(M)$ defines a global subequation, with its dual is given by $\mathcal F_{\ti \sigma,-c}$.
\end{prop}

\section{Solving the geodesic equation}
\label{proof}

We work on  $M:=A\times X$, where $A:=\{s\in\mathbb C\,|\, 1\leq |s|\leq 2\}$, and $X$ is our given compact K\"ahler manifold.  Denote by $\pi_1$ and $\pi_2$ the projections onto the first and second factor. Furthermore, let $D$ denote the differential on $M$, and $\partial$ the differential on $X$. 

Recall the space $\mathcal H$ of ``positive'' potentials for $[\alpha]$ from \eqref{space}. Given any two potentials $\phi_1,$ $\phi_2\in \mathcal H$, the goal of this section is to prove that there exists a weak, $C^0$-geodesic connecting $\phi_1$ to $\phi_2$.
\begin{thm}
\label{continousgeo}
Assume that  ${\rm osc}_X\Theta(\alpha)<\pi$, so the average angle $\hat\theta$ lifts to a real number $c$, which we assume satisfies  $(n-1)\frac\pi2<c<n\frac\pi2$. Fix $\phi_1,\phi_2\in\mathcal H$. Then there exists a continuous function $u:M\rightarrow\mathbb R$ satisfying $u(1,z)=\phi_1(z)$ and $u(2,z)=\phi_2(z)$, which is $F_{\sigma,c}$-Harmonic on $M$:
\be
\label{geos2}
u\in \mathcal F_{\sigma,c}(M)\cap-\mathcal F_{\ti\sigma,-c}(M).
\ee
Furthermore, 
\be
\label{geos3}
u(s_0,z)\in F_{\sigma, c-\frac\pi2}(X)\cap -F_{\ti\sigma, -c-\frac\pi2}(X)
\ee
for all $s_0\in A$. As a result $u$ is a weak solution to \eqref{geos}.
\end{thm}

We remark that even though $\mathcal H$ is defined for $C^\infty$ potentials, our argument only requires the potentials to be in $C^2$.

Before we prove the existence result, we first show that if  $u\in C^2(M)$, equation \eqref{geos2}  implies \eqref{geos}. By the remark following Definition \ref{harmonicdef}, \eqref{geos2}  gives $J^2u\in \partial   \mathcal F_{\sigma,c}(M).$ Now, fix a point in $x\in M$, and consider two cases, beginning with the case that $p({\rm Hess}_x\,u)\notin {\mathcal S}$. Because $\iota(\Lambda)\in\mathcal S$, we have $\iota(\Lambda)+p({\rm Hess}_x\,u)\notin\mathcal S$, and so $J^2u\in \partial   \mathcal F_{\sigma,c}(M)$ gives
\be
c=\ti\Phi(\iota(\Lambda)+p({\rm Hess}_x\,u))=\Phi(\iota(\Lambda)+p({\rm Hess}_x\,u)).\nonumber
\ee
For notational simplicity set $A_u:=\iota(\Lambda)+p({\rm Hess}_x\,u)$. Since $e^{ i\hat\theta}=e^{ i c}$, this implies 
\be
{\rm Im}\left(e^{- i\hat\theta}{\rm det}(I_{2n}+iA_u)^{\frac12}\right)=0,\nonumber
\ee
which, working in normal coordinates, we see is equivalent to \eqref{geos}. 

We now assume that  $p({\rm Hess}_x\,u)\in {\mathcal S}$, which  implies $A_u\in\mathcal S$. In this case we see right away that ${\rm det}\left(I_{2n}+i A_u\right)=0,$
which certainly implies  \eqref{geos}. Unlike the previous case, here we do not have  that  \eqref{geos} implies \eqref{geos2}. However, if in addition we assume \eqref{geos3} (or the stonger condition \eqref{geos5}),  we have
\be
\label{cosine}
{\rm cos}\left(\Theta(A_u)-\hat\theta\right)\geq0.
\ee
Although $\hat\theta$ is only defined modulo $2\pi$, we have specified a branch and thus can set $\hat\theta=c$. Because $u$ is continuous this implies
\be
\label{finalthing}
c-\frac\pi2\leq\Theta(A_u)\leq c+\frac\pi2.
\ee
Then, because $A_u\in\mathcal S$, by definition of $\ti\Phi$ we have $\ti\Phi(A_u)=\frac\pi2+\Theta(A_u)\geq c.$
So $u\in F_{\sigma,c}(M)$. Additionally, we have $\ti\Phi(-A_u)=\frac\pi 2-\Theta(A_u)\geq -c,$ which implies $u\in -F_{\ti\sigma,-c}(M)$. Thus \eqref{geos2} is satisfied.

Finally, we remark that condition is \eqref{geos3} (which is equivalent to \eqref{finalthing} for $C^2$ functions) is not quite as strong as \eqref{geos5}, since the former specifies a weak inequality while that latter a strict inequality. However, the above theorem is the best result we can prove with our methods.

Now, as stated in the introduction, our theorem is similar to Theorem 8.1 in \cite{RS}, in that we utilize Dirichlet Duality theory to solve a weak geodesic equation. However, there are some substantial differences. Most importantly, Rubinstein-Solomon work on the product of a domain in Euclidean space with an interval, and we work on $M:=X\times A$. Thus in our case $\partial M$ is smooth, whereas Rubinstein-Solomon   work on a manifold with corners.  This simplifies our boundary estimates, allowing us to appeal to general theory from \cite{HL2}. However, one key difficulty that arises is that, because our cross section $X$ is a compact manifold, there are no global sub-Harmonic functions on $M$, as one has on a domain in Euclidean space. This presents a major obstacle in proving the comparison theory, which is detailed below. To get around this difficulty, we make use of the convexity of the level set $\ti\Phi(\cdot)\geq c$, in the case that $(n-1)\frac\pi2<c<n\frac\pi2$. 

\begin{lem}
\label{convex}
The set $\cF_c:=\{A\in{\rm Sym}^2(\mathbb R^{2n+2})\,|\,\ti\Phi(A)\geq c\}$ is convex for $(n-1)\frac\pi2<c<n\frac\pi2$. 

\end{lem}
\begin{proof}
We show that ${\rm int}(\cF_c)$ is convex, since the closure of a convex set is convex. Let $A_0$ and $A_1$ be two matrices in ${\rm int}(\cF_c)$, and $A_t=(1-t)A_0+tA_1$ the path connecting these two matrices. We demonstrate that $A_t\in {\rm int}(\cF_c)$.

First, we need to work away from the set $\mathcal S$ where where the function $\ti\Phi$ is not differentiable. To accomplish this, for every matrix in the path define the perturbed matrix
\be
A_t^\delta=A_t+\delta Id_{\mathbb R^{2n+2}},\nonumber
\ee
which will be in ${\rm Sym}^2(\mathbb R^{2n+2})\backslash\mathcal S$ for $\delta>0$ small. Since $A_t^\delta\rightarrow A_t$ as $\delta\rightarrow 0$, and $\ti\Phi$ is upper-semicontinuous, we know
\be
\ti\Phi(A_t)\geq \lim_{\delta\rightarrow0} \ti\Phi(A_t^\delta).\nonumber
\ee
Thus if we can show $\ti\Phi(A_t^\delta)\in {\rm int}(\cF_c)$ for $\delta$ sufficiently small, then $\ti\Phi(A_t)$ will be on the interior as well. We now have the benefit of using the differentiable function $\Phi$, as opposed to $\ti\Phi$, since  $A_t^\delta\notin\mathcal S$.

We use the following characterization of $\Phi$, taken from \cite{RS}. Recall  the matrix $I^\eta_{2n}$ from Lemma \ref{matrixfun}, given by the diagonal $(\eta,\eta,1,1,...,1)$.  
For $A\in{\rm Sym}^2(\mathbb R^{2n+2})$, define
\be
A_\eta:=I^\eta_{2n}A I^\eta_{2n}.\nonumber
\ee
Theorem A.3 in \cite{RS} demonstrates that if $A\notin\mathcal S$, one has
\be
\Phi(A)=\lim_{\eta\rightarrow \infty}\ti\Theta(A_\eta).\nonumber
\ee
Here $\ti\Theta$ denotes the lift of the standard angle of a matrix, defined in \eqref{thetadef}.

By assumption, for small $\delta$ we have both $A_0^\delta$ and $A_1^\delta$ lie in $\in {\rm int}(\cF_c)$, so $\Phi(A_i^\delta)>c$ for $i\in\{0,1\}$. Thus we can choose $\eta_0$ sufficiently large to ensure $\ti\Theta(A_{i,\eta}^\delta)>c$ for $\eta_0<\eta$. Recall that the definition of $\ti\Theta(\cdot)$ involves first projecting the input onto the $J$-invariant subspace of symmetric matrices, where all eigenvalues have multiplicity two. Thus, even though we are working on $\mathbb R^{2n+2}$, the angle is computed from $n+1$ distinct eigenvalues.  Because $c>(n-1)\frac\pi2=((n+1)-2)\frac\pi2$, we can apply Lemma 2.1 from \cite{Y} (see also \cite{CPW}) to conclude  set  of symmetric matrices satisfying $\ti\Theta(\cdot)>c$ is convex, and therefore must contain every matrix in the path 
\be
(1-t)A_{0,\eta}^\delta+tA_{1,\eta}^\delta=A_{t,\eta}^\delta.\nonumber
\ee
Thus for all $t$,
\be
\ti\Phi(A_t)\geq \lim_{\delta\rightarrow0} \ti\Phi(A_t^\delta)=\lim_{\delta\rightarrow0}\Phi(A_{t}^\delta)=\lim_{\delta\rightarrow0}\lim_{\eta\rightarrow \infty}\ti\Theta(A_{t,\eta}^\delta)>c,\nonumber
\ee
and so ${\rm int}(\cF_c)$ is convex.
\end{proof}

The above lemma implies that the global subequation $\cF_{\sigma,c}$ is also convex for $(n-1)\frac\pi2<c<n\frac\pi2$. We need one more result before we turn to our main theorem.

\begin{prop} 
\label{comparison}
$\cF_{\sigma,c}$ satisfies the following {\bf weak comparison} condition. Given any compact set $K\subset M$, whenever $u\in \cF_{\sigma,c}^\delta(K)$ for some $\delta>0$, and $v\in\cF_{\ti\sigma,-c}(K)$,  then
\be
\label{ZMP}
u+v\leq 0\,\,\,{\rm on}\,\,\,\partial K\qquad  \implies \qquad u+v\leq 0\,\,\,{\rm on}\,\,\, K.
\ee
Here the set $\cF_{\sigma,c}^\delta(K)$ is defined as in \eqref{strictsub}.
\end{prop} 
\begin{proof}
This is a direct consequence of Theorems  8.3 and 10.1 in \cite{HL2}. The result also holds if $v$ is strictly subharmonic, as opposed to $u$. 
\end{proof}
\begin{proof}[Proof of Theorem \ref{continousgeo}.]
We employ Perron's method. Our first goal is to construct a strictly $\mathcal F_{\sigma,c}$-subharmonic function with the correct boundary data.

As a first step, we demonstrate a bound on the oscillation of the angle on each boundary slice. Given our potentials $\phi_1$ and $\phi_2$   in $\mathcal H$, consider the associated $(1,1)$ forms $\alpha_i=\alpha+i \partial\bar\partial \phi_i$ for  $i\in\{1,2\}$. By the definition of $\mathcal H$ we know
\be
\label{osc}
{\rm cos}\left(\Theta(\alpha_i)-\hat\theta\right)>0.\nonumber
\ee
Since each $\phi_i$ is continuous, the angles $\Theta(\alpha_i)$ are continuous, from which we conclude that the oscillation of $\Theta(\alpha_i)$ is bounded by $\pi$. Now, although $\hat\theta$ is only defined modulo $2\pi$, we have specified a branch  and thus can set $\hat\theta=c$.  Thus, Lemma 2.4 from \cite{CXY} gives that the angles $\Theta(\alpha_i)$ must lie in the same branch for $i=\{1,2\}$ and so
\be
|\Theta(\alpha_i)-c|<\frac\pi2.\nonumber
\ee
As a result there exists a $\delta>0$ for which
\be
\label{thing1}
c-\frac\pi2+\delta<\Theta(\alpha_i)<c+\frac\pi2-\delta.
\ee

We next consider the function 
\be
\rho:=(|s|-1)(|s|-2).\nonumber
\ee
This function is zero on  $\partial M$ and strictly negative on the interior of $M$. Furthermore
\be
 i D\overline{D}\rho=\left(1- \frac3{4|s|}\right) i ds\wedge d\bar s>\frac i4  ds\wedge d\bar s,\nonumber
\ee
which is strictly positive definite on $M$. By a slight abuse of notation we let $\phi_i:=\pi_2^*\phi_i$ be the pullback of our given functions in $\mathcal H$ to $M$. Define the function 
\be
\label{u_1}
u_1(s,z)=u_1(|s|, z):= \phi_1+\rho-C{\rm log}|s|,
\ee
where our notation specifies that $u_1$ only depends on the modulus of $s\in A$.  Choose $C$  large enough so that $u_1(2,z)<\phi_2(z)$. Because $i D\overline{D}\rho$ is positive definite, by \eqref{thing1} it follows that
\be
\ti\Phi\big(\iota(\Lambda_0)+{\rm Hess}(u_1)\big)=\frac\pi 2+ \Theta(\alpha_1) > \frac\pi2+(c-\frac\pi2+\delta)=c+\delta.\nonumber
\ee
So $u_1\in \mathcal F_{\sigma,c+\delta}(M)$. Because ${\rm Hess}(u_1)\notin\mathcal S$,  we have $\ti\Phi=\underline{\ti\Phi}$, and thus $\ti\Phi$ can not drop by $\pi/2$ with a small variation. This implies the distance between $\big(\iota(\Lambda_0)+{\rm Hess}(u_1)\big)$ and $\sim \mathcal F_{\sigma,c}$ is positive, and so $u_1$ is strictly $ \mathcal F_{\sigma,c}$-subharmonic.

Similarly we  define the function 
\be
\label{u_2}
u_2(s,z)=u_2(|s|, z):= \phi_2+\rho+A{\rm log}|s|-B,
\ee
where $A$ and $B$ are chosen so $u_2(2, z)=\phi_2$ and   $u_2(1, z)<\phi_1$. Following the same argument as above we see $u_2$ is also strictly $ \mathcal F_{\sigma,c}$-subharmonic. Now, by Lemma 7.7 in \cite{HL2}, the maximum of these two functions
\be
\underline u:=\max\{u_1,u_2\}\nonumber
\ee
is strictly $ \mathcal F_{\sigma,c}$-subharmonic as well. By construction $\underline u$ is continuous on $M$, and satisfies the correct boundary data. Define the function $\phi$ on $\partial M$ which is equal to $\phi_1$ when $|s|=1$ and $\phi_2$ when $|s|=2$, so $\underline u|_{\partial M}=\phi$.

Consider the Perron set
\be
P(\phi):=\{u\in{\rm USC}(M)\,|\,u\in \mathcal F_{\sigma,c}(M)\,{\rm and}\, u|_{\partial M}\leq \phi\}\nonumber
\ee
This set is non-empty, as it contains $\underline u$. Our next goal is to demonstrate that this set is bounded from above, allowing us to take a supremum.To accomplish this we construct a function which is strictly $ \mathcal F_{\ti\sigma,-c}$-subharmonic, and then apply the weak comparison principle from Proposition \ref{comparison}. 

Our construction is similar to that of $\underline{u}$.  On $M$ define the function 
\be
v_1(s,z)=v_1(|s|,z):= -\phi_1+\rho-C{\rm log}|s|,\nonumber
\ee
where this time $C$ is chosen so $v_1(2,z)<-\phi_2(z)$. Multiplying the inequality \eqref{thing1} by $-1$,  we can conclude
\be
\ti\Phi(-\iota(\Lambda_0)+{\rm Hess}(v_1))=\frac\pi 2- \Theta(\alpha_1) > \frac\pi2+(-c-\frac\pi2+\delta)=-c+\delta.\nonumber
\ee
Thus $v_1\in \mathcal F_{\ti\sigma,-c+\delta}(M)$, which  implies $v_1$ is strictly $ \mathcal F_{\ti\sigma,-c}$-subharmonic. Next, define
\be
v_2(s,z)=v_2(|s|,z):= -\phi_2+\rho+A{\rm log}|s|-B,\nonumber
\ee
where $A$ and $B$ are chosen so $v_2(2,z)=-\phi_2(z)$ and $v_2(1,z)<-\phi_1(z)$. In the same fashion one can show $v_2$ is strictly $ \mathcal F_{\sigma,c}$-subharmonic. Taking the maximum we see
\be
\overline{v}:=\max\{v_1,v_2\}\nonumber
\ee
is strictly $ \mathcal F_{\sigma,c}$-subharmonic. Also, by construction $\overline{v}$ is continuous on $M$, and satisfies $\overline v|_{\partial M}=-\phi$.

Choose any function $u\in P(\phi)$. Because our Perron set specifies boundary values, it is clear $u+\overline{v}\leq 0$ on $\partial M$.  By \eqref{ZMP} we conclude $u+\overline{v}\leq 0$ on $M$, and so any element of $P(\phi)$ is bounded from above by the fixed continuous function $-\overline v$. Thus the Perron function
\be
\Psi:=\sup\{u\,|\,u\in P(\phi)\}\nonumber
\ee
is well defined. To prove our Theorem, we need to demonstrate that this function $\Psi$ solves the Dirichlet problem. As a first step, we show $\Psi$ satisfies the required boundary conditions. 

We will show
\be
\label{boundaryU}
\Psi|_{\partial M}=\phi.
\ee
 Let $\Psi^*$ and $\Psi_*$ denote the upper and lower semicontinuous regularizations of $\Psi$. Because $\underline{u}\in P(\phi)$, by the definition of supremum, we know $\underline{u}\leq \Psi$. This inequality, along with the fact that $\underline{u}$ is continuous, implies  $\underline{u}= \underline{u}_*\leq \Psi_*$. In particular, on $\partial M$ we have $\phi\leq \Psi_*$. Now, we have argued above that, for any $u\in P(\phi)$, by weak comparison $u+\overline{v}\leq0$ on $M$. Taking the supremum over $P(\phi)$ yields $\Psi\leq -\overline v$. Since $\overline v$ is continuous, $\Psi^*\leq -\overline {v}^*=-\overline v. $ Thus on $\partial M$ we have $\Psi^*\leq \phi$. Putting everything together, on $\partial M$ we have
\be
\phi\leq \Psi_*\leq\Psi\leq \Psi^*\leq\phi\nonumber
\ee
which implies \eqref{boundaryU}.

Our next goal is to show $\Psi$ is continuous. By Theorem 2.6 in \cite{HL2}, the upper semicontinuous regularization $\Psi^*$ of $\Psi$ belongs to $\mathcal F_{\sigma,c}(M)$.  Yet we have just demonstrated $\Psi^*|_{\partial M}=\phi$, and  so $\Psi^*\in P(\phi)$. By the definition of supremum this implies $\Psi^*\leq \Psi$, and thus $\Psi=\Psi^*$. To prove continuity, it remains to show that $\Psi_*=\Psi$.

As a first step, we argue that $-\Psi_*\in\mathcal F_{\ti\sigma,-c}$. This follows from Lemma $\ti{\bf F}$ in \cite{HL2}, although we include the details here for the reader's convenience. Via contradiction, assume there exists a point $p\in M$, and function $h\in C^2$, satisfying $-\Psi_*(p)=h(p)$ and 
\be
\label{thing2}
-\Psi_*-h\leq -\epsilon |y-p|^2
\ee
(in local coordinates $y$ near $p$), while at the same time
\be
J^2_ph\notin(\mathcal F_{\ti\sigma,-c})_p.\nonumber
\ee
Since $\mathcal F_{\ti\sigma,-c}$ is a closed set, we have $-J^2_ph\in {\rm Int}( {\mathcal F}_{\sigma,c})_p$. Thus there exists small $r>0$ and $\delta>0$, so that $h_\delta:=-h+\delta$
is ${\mathcal F}_{\sigma,c}$-subharmonic on $B(p,r)$. Moreover, \eqref{thing2} implies that for $\delta$ small enough, $h_\delta<\Psi_*$ in a neighborhood of $\partial B(p,r)$. Since $\Psi_*\leq \Psi$, this implies that the function
\be
h' := \begin{cases}\Psi&\mbox{on } M-B(p,r) \\
\max\{\Psi,h_\delta\} &\mbox{on } \overline{B(p,r)}. \end{cases}\nonumber
\ee
is  ${\mathcal F}_{\sigma,c}$-subharmonic on all of $M$. Because $h'=\Psi=\phi$ on $\partial M$, we have $h'\in P(\phi)$. By definition of supremum it follows that $h'\leq \Psi$ on $M$, and in particular this implies $h_\delta\leq \Psi$ on $B(p,r)$. 

Now, we have already assumed $-\Psi_*(p)=h(p)$, which implies $\Psi_*(p)=h_\delta(p)-\delta$. By definition of the lower semicontinuous regularization
\be
\Psi_*(p)=\liminf_{y\rightarrow p} \Psi(y),\nonumber
\ee
so there exists a sequence of points $y_k$ such that $\Psi(y_k)\rightarrow \Psi_*(p)$. Thus $\Psi(y_k)\rightarrow h_\delta(p)-\delta$, while at the same time, because $h_\delta\in C^2$,  $h_\delta(y_k)\rightarrow h_\delta(p)$. This implies for large enough $k$ that $h_\delta(y_k)>\Psi(y_k)$, a contradiction.  Thus $-\Psi_*\in\mathcal F_{\ti\sigma,-c}(M)$.

Our next goal is to apply the comparison principle to prove continuity of $\Psi$. Specifically, $\Psi-\Psi_*=0$ on $\partial M$, and $\Psi\in \mathcal F_{\sigma,c}(M)$, so naively one may hope comparison holds right away to conclude $\Psi-\Psi_*\leq 0$ on $M$, proving $\Psi\leq \Psi_*$. Unfortunately, in our setting, comparison only holds for either strictly-$\mathcal F_{\sigma,c}$, or strictly-$\mathcal F_{\ti\sigma,-c}(M)$  subharmonic functions, which neither $\Psi$ nor $\Psi_*$ are.  Thus we first approximate $\Psi$ uniformly by a sequence of strictly subharmonic functions, and then apply Proposition \ref{comparison} to this sequence.  Because our fibers over $A$ are the compact manifold $X$, there are no global subharmonic functions on $M$ with which one can perform this approximation. This is in contrast to a domain in $\mathbb C^n$, where, for example, one can use $\epsilon |z|^2$. Instead we construct our approximation  using the convexity of the level sets of $\Phi(\cdot)$. As mentioned above this is one of the key differences between our proof and the approach in \cite{RS}.

In order to construct our uniform approximation, we use the functions $u_1$ and $u_2$ from \eqref{u_1} and \eqref{u_2}. We show that, for each $i$ and all $\epsilon>0$, 
\be
\Psi_{i,\epsilon}:=(1-\epsilon)\Psi+\epsilon u_i\nonumber
\ee
is a strictly $\mathcal F_{\sigma,c}$-subharmonic function. By definition, since $\Psi$ is only upper-semicontinuous, we need to verify that for any $C^2$ function $f_\epsilon$ which satisfies $f_\epsilon(p)=\Psi_{i,\epsilon}(p)$, and $f_\epsilon\geq\Psi_{i,\epsilon}$ near $p$, one has  $J^2_p(f_\epsilon)\in \mathcal F_{\sigma,c}^\delta$ for some $\delta>0$. Define
\be
 f:=\frac{f_\epsilon-\epsilon u_i}{1-\epsilon},\nonumber
\ee
which is also a $C^2$ function. We see that $f_\epsilon$ touches the graph of $\Psi_{i,\epsilon}$ from above at the point $p$ if and only if $ f$ touches the graph of $\Psi$ from above at $p$. Since  $\Psi\in \mathcal F_{\sigma,c}(M)$, by definition $J^2_p(f)\in \mathcal F_{\sigma,c}$. Furthermore, by construction $J^2_p(u_i)\in{\rm int}\left( \mathcal F_{\sigma,c}\right).$ Now, recall $ \mathcal F_{\sigma,c}$ is locally jet equivalent to the subequation
\be
\cF_c:=\{A\in{\rm Sym}^2(\mathbb R^{2n+2})\,|\,\ti\Phi(A)\geq c\}.\nonumber
\ee
By Lemma \ref{convex}, the level sets of  $\ti\Phi(\cdot)$ are convex, and since $f_\epsilon$ is on the path connecting $f$ to $u_i$,  we conclude $J^2_p(f_\epsilon)\in {\rm int}(\mathcal F_{\sigma,c})$ for all $0<\epsilon\leq1$. Thus, we can always find a $\delta>0$ so that  $J^2_p(f_\epsilon)\in \mathcal F_{\sigma,c}^\delta$, proving $\Psi_{i,\epsilon}$ is strictly $\mathcal F_{\sigma,c}$-subharmonic for $i\in\{1,2\}.$

Define the function $\Psi_\epsilon:=\max\{\Psi_{1,\epsilon},\Psi_{2,\epsilon}\}$, which by the definition of $\underline{u}$ is equal to $(1-\epsilon)\Psi+\epsilon\underline{u}$. Thus, not only does $\Psi_\epsilon$ satisfy the boundary condition $\Psi_\epsilon|_{\partial M}=\phi$, but because $\underline u$ is continuous and $M$ is compact, for any $\epsilon'>0$, we can choose $\epsilon$ small enough to guarantee 
\be
\Psi-\epsilon'\leq \Psi_\epsilon\leq \Psi+\epsilon'.\nonumber
\ee
On $\partial M$, we then have
\be
\Psi_\epsilon-\epsilon'-\Psi_*\leq\Psi-\Psi_*\leq 0.\nonumber
\ee
 By Lemma 7.7 in \cite{HL2}, $\Psi_\epsilon$ is strictly $\mathcal F_{\sigma,c}$-subharmonic. Since subtracting a small constant does not change the Hessian,  $\Psi_\epsilon-\epsilon'$ is strictly $\mathcal F_{\sigma,c}$-subharmonic as well. We now apply the weak comparison principle, Proposition \ref{comparison}, to conclude
\be
\Psi_\epsilon-\epsilon'\leq \Psi_* \,\,\,{\rm on}\,\,\, M.\nonumber
\ee
Yet this implies $\Psi-2\epsilon'\leq\Psi_*$, for any $\epsilon'>0$. Thus $\Psi\leq\Psi_*$.

This demonstrates that $\Psi_*=\Psi=\Psi^*$, and so $\Psi$ is continuous. Furthermore $\Psi$ satisfies the boundary condition $\Phi|_{\partial M}=\phi$, as well as 
\be
\Psi\in \mathcal F_{\sigma,c}(M)\cap-\mathcal F_{\ti\sigma,-c}(M).\nonumber
\ee
Thus $\Psi$ solves the Dirichlet problem \eqref{geos2}.

To conclude, we demonstrate that $\Psi$ satisfies \eqref{geos3}. In fact, this follows from comparing the space-time angle of a given matrix to the standard angle along space-like slices. Specifically, given any $A\in {\rm Sym}^2(\mathbb R^{2n+2})\backslash \mathcal S$, using the matrix $B^{\mathbb C}$ from the proof of Lemma \ref{matrixfun} (with $\eta=0$), the associated angles can be expressed as
\be
\Phi(A):={\rm arg\,det}(B^{\mathbb C})\qquad{\rm and}\qquad \Theta(A^+)={\rm arg\,det}(B^{\mathbb C}_+).\nonumber
\ee
Using equation \eqref{detB}, we can write the difference as
\bea
\Phi(A)-\Theta(A^+)&=&{\rm arg}\left({\rm det}(B^{\mathbb C}){\rm det}(B^{\mathbb C}_+)^{-1}\right)\nonumber\\
&=&{\rm arg}\left(i{\bf a}_{11}+\vec{\bf a}_1 (B_+^{\mathbb C})^{-1} {\vec{ \bf a}}_1^*\right).\nonumber
\eea
Since $i{\bf a}_{11}+\vec{\bf a}_1 (B_+^{\mathbb C})^{-1} {\vec{ \bf a}}_1^*$ has non-negative real part, we conclude
\be
\Phi(A)-\Theta(A^+)\in[-\frac\pi2,\frac\pi2].\nonumber
\ee
Following Lemma 3.7 in \cite{RS}, using upper semi-continuity along a path of matrices, we can prove that lifted angle $\ti\Phi$ satisfies
\be
|\ti \Phi(A)-\Theta(A^+)|\leq\frac\pi2,\nonumber
\ee
and this extends to the case that $A\in \mathcal S$.

Thus, touching the graph of $\Psi$ above and below by $C^2$ functions, and computing the associated angles $\ti \Phi$, the above angle bounds easily give the desired control of the standard angle $\Theta$, demonstrating that $\Psi$ satisfies \eqref{geos3}.

\end{proof}

\end{normalsize}

\end{document}